\documentclass[12pt]{amsart}
\usepackage{
	amsmath,  amssymb,  amsthm,   amscd,
	gensymb,  graphicx, etoolbox, booktabs,
	stackrel, mathtools    
}
\usepackage[usenames,dvipsnames]{xcolor}
\definecolor{darkblue}{rgb}{0,0,0.7}
\definecolor{darkred}{rgb}{0.7,0,0}
\usepackage[colorlinks=true, linkcolor=darkred, citecolor=darkblue, urlcolor=blue, pagebackref=true, breaklinks=true]{hyperref}
\usepackage[capitalise]{cleveref}
\usepackage{enumitem}
\usepackage{placeins}
\usepackage{relsize}
\usepackage{todonotes}
\usepackage{soul}
\usepackage{tikz}
\usetikzlibrary{arrows,shapes}
\usepackage{float}
\usepackage{tabularx}
\usepackage{kantlipsum}
\usepackage{array}
\usepackage[margin=2.5cm]{geometry}

\newtheorem{proposition}{Proposition}[section]
\newtheorem{theorem}[proposition]{Theorem}
\newtheorem{lemma}[proposition]{Lemma}
\newtheorem{corollary}[proposition]{Corollary}
\newtheorem{question}[proposition]{Question}
\newtheorem{conjecture}[proposition]{Conjecture}
\newtheorem{remark}[proposition]{Remark}

\newtheorem{definition}[proposition]{Definition}

\newtheorem{innercustomthm}{Theorem}
\newenvironment{customthm}[1]
{\renewcommand\theinnercustomthm{#1}\innercustomthm\itshape}
{\endinnercustomthm}

\newcommand{\G}{{\mathcal{G}}}

\def\H{\mathcal{H}}

\tikzstyle{place}=[draw,circle,minimum size=1mm,inner sep=1pt,outer sep=-1.1pt,fill=black]

\tikzstyle{places}=[draw,rectangle,minimum size=8pt,inner sep=0pt]
\tikzstyle{placesf}=[draw,rectangle,minimum size=5pt,inner sep=0pt]
\tikzstyle{placec}=[draw,circle,minimum size=8pt,inner sep=0pt]
\tikzstyle{placecf}=[draw,circle, minimum size=5pt,inner sep=0pt]

\def\H{\mathcal{H}}
\def\K{\mathbb{K}}
\def\G{\mathcal{G}}

\def\p{\mathfrak {p}}

\def\E{\mathcal{E}}

\def\l{\langle}
\def\ha{\widehat}
\def\til{\widetilde}
\def\r{\rangle}
\def\x{\mathbf x}

\def\height{\mathrm{ht}}

\title[Ordinary and symbolic equality]{Equality of ordinary and symbolic powers and the Conforti-Cornu\'ejols conjecture for $(n-2)$-uniform clutters}

\author{Amit Roy}
\address{Chennai Mathematical Institute, India}
\email{amitiisermohali493@gmail.com}

\author{Kamalesh Saha}
\address{Department of Mathematics, SRM University-AP, Amaravati 522240, Andhra Pradesh, India}
\email{kamalesh.saha44@gmail.com; kamalesh.s@srmap.edu.in}

\thanks{}

\begin{document}
	
	\begin{abstract}
		 Let $I$ be an equigenerated squarefree monomial ideal in the polynomial ring $\mathbb{K}[x_1,\ldots,x_n]$, and let $\mathcal{H}$ be a uniform clutter on the vertex set $\{x_1,\ldots,x_n\}$ such that $I=I(\mathcal{H})$ is its edge ideal. A central and challenging problem in combinatorial commutative algebra is to classify all clutters $\mathcal{H}$ for which $I(\mathcal{H})^{(k)} = I(\mathcal{H})^{k}$ for a fixed positive integer $k$, where $I(\mathcal{H})^{(k)}$ denotes the $k^{\text{th}}$ symbolic power of $I(\mathcal{H})$. 
		
		In this article, we give a complete solution to this problem for $(n-2)$-uniform clutters. Moreover, we provide a simple combinatorial classification of all $(n-2)$-uniform clutters having the packing property. As a consequence, we confirm the celebrated Conforti–Cornuéjols conjecture for $(n-2)$-uniform clutters. We also compare our results with the known families of clutters for which the conjecture is known to be true. Finally, we present an application of our results to the theory of Linear Programming duality problems.
	\end{abstract}
	
	\maketitle
	
	\section{Introduction}\label{intro}
	The study of symbolic powers of ideals has historically played a central role in commutative algebra. For instance, symbolic powers featured prominently in the classical proof of Krull's principal ideal theorem, and they were also used in the proof of the Hartshorne–Lichtenbaum vanishing theorem. More generally, symbolic powers capture the intricate algebraic and geometric constraints imposed by the minimal primes of an ideal, revealing subtleties that ordinary powers alone cannot detect. A central problem in this area concerns understanding the containment relations between symbolic and ordinary powers. In particular, significant attention has been devoted to characterizing ideals for which symbolic powers coincide with ordinary powers. This question has been extensively studied due to its deep connections with combinatorics, homological invariants, and the geometry of the underlying varieties (see, for instance, \cite{DaoDeGriHuBet2018, SzembergSzpond2017}).
	
	In the case of squarefree monomial ideals, the study of symbolic powers is closely related to the so-called packing problem. Let $I$ be a squarefree monomial ideal in the polynomial ring $S=\mathbb K[x_1,\ldots,x_n]$, where $\K$ is a field. Let $\H$ denote a clutter with vertex set $V(\H)=\{x_1,\ldots,x_n\}$ and edge set $\E(\H)$. The ideal $I$ can be expressed as the edge ideal 
	\[
	I(\H)=\langle \x_{\E}:=\prod_{x_i\in \E}x_i\mid \E\in E(\H)\rangle.
	\]
	We say that $\H$ satisfies the \emph{K\"onig property} if the height of $I(\H)$ equals the maximum number of independent edges of $\H$ (known as the {\it matching number} of $\H$). A clutter $\H$ is said to satisfy the \emph{packing property} if all its minors satisfy the K\"onig property. 
	
	In 1993, Conforti and Cornu\'ejols~\cite{ConfortiCornuejols1993} proposed a conjecture asserting that a clutter satisfies the packing property if and only if it satisfies the \emph{max-flow min-cut} (MFMC) property. Recall that a clutter $\H$ is said to satisfy the MFMC property if both sides of the Linear Programming duality equation 
	\[
	\min\{\alpha\cdot x\mid Ax\ge\mathbf 1\}=\max\{y\cdot \mathbf 1\mid yA\le\alpha,y\in\mathbb N^m\}
	\]
    have integral optimum solutions $x$ and $y$ for each nonnegative integral vector $\alpha$, where $\mathbf 1$ is the column vector with all entries $1$, and $A$ is the (edge-vertex) incidence matrix of $\H$ with $|V(\H)|=n$ and $|E(\H)|=m$.
	\color{black}
	 The Conforti--Cornu\'ejols conjecture, together with the underlying combinatorial notions, was subsequently translated into the framework of commutative algebra by Gitler, Valencia, and Villarreal~\cite{Gitler2007Rees}. In particular, the commutative algebraic formulation of the conjecture, as stated in~\cite[Conjecture~3.10]{Gitler2007Rees} (see also~\cite{grv2009}), is the following.

	\begin{conjecture}[Conforti–Cornu\'ejols Conjecture]\label{packing conjecture}
		Let $\mathcal{H}$ be a clutter, and let $I(\H)$ denote its edge ideal. Then the following statements are equivalent:
		\begin{enumerate}[label=(\roman*)]
			\item $I(\H)^{(k)} = I(\H)^k$ for all $k \ge 1$;
			\item $\mathcal{H}$ satisfies the packing property. 
		\end{enumerate}
	\end{conjecture}
	
	Over the past three decades, there has been substantial progress toward verifying this conjecture. It is known to hold for several important classes of clutters, including balanced clutters \cite{FulkersonHoffmanOppenheim1974}, binary clutters \cite{Seymour1977}, dyadic clutters \cite{CornuejolsMargotGuenin2000}, and a family of $2$-partitionable clutters \cite{FloresMendezGitlerReyes2008}. In the context of edge ideals of clutters, the conjecture is verified in a few special cases, including edge ideals of graphs \cite{SimisVasconcelosVillarreal1994}, $3$-path ideals of graphs \cite{AlilooeeBanerjee2021}, $4$-path ideals of graphs \cite{HochstattlerNasernejad2024}, and the matroidal ideals \cite{FicarraMoradi2025SymbolicPowersPolymatroidal}.
	
	In this direction, a closely related and more challenging problem is the following.
	
	\begin{question}\label{main question}
		Let $\mathcal{H}$ be a clutter, and let $I(\H)$ denote its edge ideal. For a fixed positive integer $k$, can one provide a combinatorial classification of all clutters $\H$ for which $I(\H)^{(k)} = I(\H)^k$?
	\end{question}

	Following~\cite{Mendez2025Symbolic}, a monomial ideal \( I \) is said to be \emph{Simis in degree} \( k \) if \( I^{(k)} = I^k \) for a fixed positive integer \( k \). The question posed above thus seeks a combinatorial characterization of clutters whose edge ideals are Simis in degree \( k \). This problem remains largely open even within the class of \( d \)-uniform clutters. To the best of our knowledge, complete answers are known only in the cases \( d = 1,2,n \). Indeed, when \( d = 1 \), each vertex of the clutter \( \mathcal{H} \) forms an edge, and consequently \( I(\mathcal{H}) \) is generated by variables. On the other hand, when \( d = n \), the clutter \( \mathcal{H} \) consists of a single edge, making \( I(\mathcal{H}) \) a principal ideal. Hence, in both cases, \( I(\mathcal{H}) \) is Simis in degree \( k \) for every \( k>0 \). In the case of $d=2$, a combinatorial classification in terms of the associated graph follows from \cite[Lemma 5.8 and Theorem 5.9]{SimisVasconcelosVillarreal1994} (see also \cite[Lemma 3.10]{RTY11}). For a general \( d \), a necessary condition for the equality, expressed in terms of the dual ideal, was established in~\cite[Theorem~6.1]{Mendez2025Symbolic}.
	
	Our main result of this paper provides a complete answer to \cref{main question} and a verification of \cref{packing conjecture} for the $(n-2)$-uniform clutters. Our approach is based on expressing the edge ideal of such a clutter as the complementary edge ideal of an associated graph (see \Cref{section 2} for the definition). Recall that a clutter $\H$ is said to be \emph{$d$-uniform} if $|\E|=d$ for every $\E\in E(\H)$. The $2$-uniform clutters are well-known in the literature as graphs, while an $(n-2)$-uniform clutter $\H$ can be naturally associated with a graph $G_{\H}$ defined in the following way:
	\[
	V(G_{\H})=V(\H), \qquad
	E(G_{\H})=\{V(\H)\setminus \E \mid \E\in E(\H)\}.
	\]
	Under this correspondence, $I(\H)=I_c(G_{\H})$, where $I_c(G_{\H})$ denotes the \emph{complementary edge ideal} of $G_{\H}$. The notion of the complementary edge ideal of a graph was introduced independently and almost simultaneously by Hibi-Qureshi-Madani~\cite{HibiQureshiMadani2025} and Ficarra-Moradi~\cite{FicarraMoradiComp2025}. Although a broader version for clutter had earlier appeared in Villarreal’s monograph~\cite[Definition~14.6.23]{RHV}, where it is referred to as the dual of the edge ideal of a clutter. The complementary edge ideal of a graph has been the subject of considerable recent attention; see, for instance,~\cite{FicarraMoradiComp2025, FicarraMoradiRees2025, FicarraMoradiSR2025, HibiQureshiMadani2025}. These studies reveal deep combinatorial patterns governing the algebraic properties of the corresponding ideals and their powers. Using the language of complementary edge ideal, we provide a complete solution to \Cref{main question} and confirm the Conforti–Cornu\'ejols Conjecture for $(n-2)$-uniform clutters in the following theorem.

	\begin{customthm}{\ref{main theorem12345}}
		Let $\H$ be an $(n-2)$-uniform clutter with $|V(\H)|=n$. Let $G_{\H}$ denote the associated graph such that $I(\H)=I_c(G_{\H})$. Then the following are equivalent.
		\begin{enumerate}
			\item $I(\H)^{(k)}=I(\H)^k$ for all $k\geq 1$;
			\item $I(\H)^{(k)}=I(\H)^k$ for some $k\geq 2$;
			\item $\H$ satisfies the MFMC property; 
			
			\item $\H$ satisfies the packing property;
			
			\item $G_{\H}$ is either a $K_2$, $K_3$, $P_3$, $2K_2$, $P_4$ or $C_4$ with (possibly) some isolated vertices.
		\end{enumerate}
	\end{customthm}
	%where $K_m,P_m,C_m$ denote the complete graph, path graph and the cycle graph on $m$ vertices, respectively, and $2K_2$ denote the graph on $4$ vertices with $2$ disjoint edges.
	In the next section, we show that the class of $(n-2)$-uniform clutters is not properly contained in any of the classes of clutters mentioned above for which the Conforti-Cornu\'ejols conjecture is known to hold. More precisely, for each of these classes, we construct $(n-2)$-uniform clutters that do not belong to that class; that is, for the classes of balanced clutters, binary clutters, dyadic clutters, $2$-partitionable clutters considered in \cite{FloresMendezGitlerReyes2008}, and the clutters associated with matroidal ideals, we exhibit examples of $(n-2)$-uniform clutters which fall outside each respective class. Note that for each $n \ge 7$, no $(n-2)$-uniform clutter can be realized as the $3$-path ideal or the $4$-path ideal of a graph. Furthermore, by \cite[Theorem 3.1]{FicarraMoradiComp2025}, it follows that the edge ideal of an $(n-2)$-uniform clutter is, in almost all cases, not a matroidal ideal.

	It is worth noting that the result stated in \Cref{main theorem12345} goes beyond a mere verification of the \Cref{packing conjecture} for $(n-2)$-uniform clutters. In fact, there exist examples of $2$-uniform clutters (graphs) $G$ for which $I(G)^{(k)} = I(G)^k$ holds for some fixed positive integer $k\geq 2$, even though $G$ fails to satisfy the MFMC property (see \Cref{last remark}).

    An important application of our result concerns a Linear Programming duality problem. In fact, as a consequence of \Cref{main theorem12345}, we establish in \Cref{lpp theorem} a necessary and sufficient condition for a certain $(0,1)$-matrix to admit an integral optimal solution in the corresponding Linear Programming duality equation.

	\section{Main results}\label{section 2}
	In this section, we present the proof of \Cref{main theorem12345}. In particular, for the proof of \Cref{main theorem12345}, we need to analyze the packing property of certain classes of clutters. To this end, we recall the following combinatorial notions:
	
Let \( \mathcal{H} \) be a clutter with the vertex set \( V(\H)\) and the edge set \( E(\mathcal{H}) \), and let $x_i\in V(\H)$.

\begin{itemize}
	\item \textbf{Contraction} \( \mathcal{H}/x_i \): The contraction of a clutter \( \mathcal{H} \) at \( x_i \), denoted by \( \mathcal{H}/x_i \), is the clutter with the vertex set $V( \mathcal{H}/x_i)=V(\H)\setminus \{x_i\}$ and the edge set $E(\mathcal{H}/x_i)$ consisting of the minimal elements of the set $\{ e \setminus \{x_i\} \mid e \in E(\mathcal{H}) \}$, where minimality is with respect to inclusion.
	
	\item \textbf{Deletion} \( \mathcal{H} \setminus x_i \): The deletion of a clutter \( \mathcal{H} \) at \( x_i \), denoted by \( \mathcal{H} \setminus x_i \), is the clutter with $V( \mathcal{H}\setminus x_i)=V(\H)\setminus \{x_i\}$ and $E(\mathcal{H}\setminus x_i)=\{e\in E(\H)\mid x_i\notin e\}$.
	
	\item \textbf{Minor}: A minor of a clutter \( \mathcal{H} \) is any clutter obtained from \( \mathcal{H} \) by a sequence of deletions and contractions, performed in any order.
		\end{itemize}
		
		As mentioned in the introduction, we say that $\H$ satisfies the \emph{K\"onig property} if the height of $I(\H)$ equals the {\it matching number} of $\H$. A clutter $\H$ is said to satisfy the \emph{packing property} if all its minors satisfy the K\"onig property. For definitions of matching number, vertex cover, and other related concepts, we advise the reader to the book by Villarreal \cite{RHV}.
		
		Let $\H$ be a clutter with vertex set $V(\H)=\{x_1,\ldots,x_n\}$ and edge set $E(\H)$. Let $\{y_1,\ldots,y_r\}$ be a new set of vertices. We define a new clutter $\H_{y_1\cdots y_r}$ as follows:
		\begin{align*}
			V(\H_{y_1\cdots y_r})&=V(\H)\cup\{y_1,\ldots,y_r\}\\
			E(\H_{y_1\cdots y_r})&=\{e\cup\{y_1,\ldots,y_r\}\mid e\in E(\H)\}.
		\end{align*}
		 In the following lemma we discuss the relationship of the packing property for $\H$ and $\H_{y_1\cdots y_r}$.
		
		\begin{lemma}\label{packing lemma}
			Let $\H$ and $\H_{y_1\cdots y_r}$ be two clutters defined as above. Then $\H$ satisfies the packing property if and only if $\H_{y_1\cdots y_r}$ satisfies the packing property.
		\end{lemma} 
		\begin{proof}
			Enough to assume $r=1$. It is easy to see that if $\H_{y_1}$ satisfies the packing property, then $\H$ also satisfies the packing property. Indeed, if $\H'$ is a minor obtained from $\H$ by a sequence of deletions and contractions of vertices $x_{i_1},\ldots,x_{i_t}$, then $\H'$ can also be obtained from $\H_{y_1}$ by first contracting the vertex $y_1$, and then applying the same sequence of deletions and contractions of vertices $x_{i_1},\ldots,x_{i_t}$. Thus, $\H'$ has the K\"onig property, and consequently, $\H$ has the packing property.
			
			Conversely, suppose $\H$ satisfies the packing property and $\H_{y_{1}}'$ is a minor obtained from $\H_{y_{1}}$ by a sequence of deletions and contractions of vertices $v_{1},\ldots,v_{s}$. Observe that if $v_i\neq y_1$ for each $i\in [s]$, then 
			\[
			\mathrm{ht}(I(\H_{y_{1}}'))=\mathrm{mat}(\H_{y_{1}}')=1,
			\]
			where $\mathrm{ht}$ denotes the height of the ideal, and $\mathrm{mat}$ denotes the matching number of the clutter. Thus, $\H_{y_{1}}'$ satisfies the K\"onig property. Now, suppose $v_i=y_1$ for some $i\in [s]$. Then two possible cases arise. If the role of $y_1$ is deletion in the sequence, then $\H_{y_{1}}'$ is a clutter with no edges, thus trivially satisfies the K\"onig property. If the role of $y_1$ is contraction in the sequence, then $\H_{y_{1}}'$ is the same as the minor obtained from $\H$ by performing the sequence of deletions and contractions of vertices $v_1,\ldots,\widehat{v_i},\ldots,v_s$ (in the same order). Thus, $\H_{y_{1}}'$ satisfies the K\"onig property, and consequently, $\H_{y_{1}}$ satisfies the packing property, as desired.
		\end{proof}

	The edge ideal of an $(n-2)$-uniform clutter on the vertex set $\{x_1,\ldots,x_n\}$ can be realized as the complementary edge ideal of an associated graph $G_{\H}$. Consequently, to verify the equality $I(\H)^{(k)} = I(\H)^k$ for an $(n-2)$-uniform clutter, it suffices to study the ordinary and symbolic powers of the complementary edge ideal of $G_{\H}$. For notational simplicity, we shall denote the graph by $G$ and its complementary edge ideal by $I_c(G)$. By definition, $I_c(G)=\left\l \x_{\bar e}\mid e\in E(G) \right\r$ is a squarefree monomial ideal in $R=\mathbb K[x_1,\ldots,x_n]$, where $\bar e=V(G)\setminus e$ and $\x_{\bar e}=\prod_{x_i\in \bar e}x_i$..
	
	For a finite simple graph $G$, the primary decomposition of $I_c(G)$ admits a combinatorial description, which we present below. Throughout, we adopt the convention that for any subset $A \subseteq \{x_1,\ldots,x_n\}$, the prime ideal $\p_A$ is generated by the variables in $A$.
	
	\begin{proposition}\label{primary decomposition}
		Let $G$ be a finite simple graph with at least one edge. Then
		\[
		I_c(G)
		= \Big(\bigcap_{x \text{ isolated in } G} \p_{\{x\}}\Big)
		\;\cap\;
		\Big(\bigcap_{e \in E(G^c)} \p_e\Big)
		\;\cap\;
		\Big(\bigcap_{G[A] \cong K_3} \p_A\Big).
		\] 
	\end{proposition}
	\begin{proof}
		Since $G$ contains at least one edge, it follows that $I_c(G)\neq (0)$, and thus $\G(I_c(G))\neq\emptyset$. Moreover, if $x$ is an isolated vertex of $G$, then for each $m\in \G(I_c(G))$ we have $x\mid m$, where $\G(I)$ denotes the unique minimal monomial generating set of a monomial ideal $I$. Thus, $\p_{\{x\}}$ is a minimal prime ideal of $I_c(G)$. The result now follows from \cite[Theorem 2.1]{FicarraMoradiComp2025} (see also \cite[Theorem 1.1]{HibiQureshiMadani2025}).
	\end{proof}
	
	\begin{remark}
		Let $\H$ be the $(n-2)$-uniform clutter such that $I(\H)=I_c(G)$. Then the minimal prime ideal of $I_c(G)$ corresponds to the minimal vertex covers of $\H$.
	\end{remark}
	
	Before going to our main results, let us recall the definition of symbolic powers.
	
	\begin{definition}[Symbolic power]\label{def-symbolic}
		{\rm Let $R$ be a Noetherian ring and $I\subset R$ be an ideal. Then, the  $k^{th}$ symbolic power of $I$ has the following two definitions.
			\begin{enumerate}
				\item $ \displaystyle I^{(k)}=\bigcap_{P\in \mathrm{Ass}(R/I)}(I^k R_P \cap R).$
				\item $\displaystyle I^{(k)}=\bigcap_{P\in \mathrm{Min}(I)}(I^k R_P \cap R),$  where $\mathrm{Min}(I)$ is the set of all minimal prime of $I$.
		\end{enumerate}}
	\end{definition}
	
	Note that for a squarefree monomial ideal $I$, the two definitions of symbolic power discussed above coincide, as $I$ has no embedded primes. More precisely, for a squarefree monomial ideal $I$, we have $
	I^{(k)} = \bigcap_{\p \in \mathrm{Ass}(I)} \p^k$,
	where $\mathrm{Ass}(I)$ denotes the set of associated primes of $I$. Consequently, using \Cref{primary decomposition}, we obtain
	\[
	I_c(G)^{(k)}
	= \Big(\bigcap_{x \text{ isolated in } G} \p_{\{x\}}^k\Big)
	\;\cap\;
	\Big(\bigcap_{e \in E(G^c)} \p_e^k\Big)
	\;\cap\;
	\Big(\bigcap_{G[A] \cong K_3} \p_A^k\Big),
	\]
	for each positive integer $k$.
	
	The following result addresses the equality of symbolic and ordinary powers of complementary edge ideals in the presence of isolated vertices in the underlying graph.
	
	\begin{proposition}\label{isolated lemma}
		Let $G$ be a graph and $x$ an isolated vertex of $G$. Let $\ha G=G\setminus x$. Then for each positive integer $k$, $I_c(G)^{(k)}=I_c(G)^{k}$ if and only if $I_c(\ha G)^{(k)}=I_c(\ha G)^{k}$.
	\end{proposition}
	\begin{proof}
		If $G$ does not contain any edge, then $I_c(G)=I_c(\ha G)=\l 0\r$, and we clearly have the result. Therefore, we may assume $G$ contains at least one edge. By definition of symbolic powers, we have $I_c(G)^{k}\subseteq I_c(G)^{(k)}$ and $I_c(\ha G)^{k}\subseteq I_c(\ha G)^{(k)}$.
		
		Suppose $I_c(G)^{k}=I_c(G)^{(k)}$ and take $m\in \G(I_c(\ha G)^{(k)})$. Then, by \Cref{primary decomposition}, $x^k\,m\in I_c(G)^{(k)}=I_c(G)^{k}$. Therefore, there exists a monomial $m'$ in $R$ such that 
		\[
		x^k\,m=m'\,\prod_{i=1}^k\x_{\bar e_i},
		\]
		where $e_1,\ldots,e_k\in E(G)$. Observe that $x\in \bar e_i$ for each $i\in [k]$, where $\bar e_i=V(G)\setminus e_i$. Hence, we can write $x^km=x^km'\prod_{i=1}^k\x_{\bar f_i}$, where $f_i=e_i\setminus\{x\}\in E(\ha G)$ for each $i\in [k]$ and $\bar f_i=V(\ha G)\setminus f_i$. In other words, $m\in I_c(\ha G)^k$, as required.
		
		Next, suppose $I_c(\ha G)^{k}=I_c(\ha G)^{(k)}$, and take $m\in \G(I_c(G)^{(k)})$. Then, again by \Cref{primary decomposition}, $m=x^km_1$, where $m_1\in I_c(\ha G)^{(k)}=I_c(\ha G)^{k}$. Therefore, $m_1=m_1'\prod_{i=1}^k\x_{\bar g_i}$, where $g_i\in E(\ha G)$ for each $i\in [k]$ and $\bar g_i=V(\ha G)\setminus g_i$. Note that $x\,\x_{\bar g_i}\in I_c(G)$ for each $i\in [k]$. Thus $m=m_1'x^k\prod_{i=1}^k\x_{\bar g_i}\in I_c(G)^{k}$, and this completes the proof of the proposition.
	\end{proof}
	
	\begin{remark}
		After a preliminary version of this article was put into arXiv, it was informed to us by A. Ficarra that the result in \Cref{isolated lemma} can also be deduced from \cite[Proposition 6.3]{FicarraMoradi2025SymbolicPowersPolymatroidal}. However, we keep the proof here for the sake of completeness.
	\end{remark}
	\begin{remark}
		In view of \Cref{isolated lemma}, it follows that to verify the equality $I_c(G)^{(k)}=I_c(G)^{k}$ for a graph $G$, it suffices to restrict attention to graphs without isolated vertices.
	\end{remark}
	
	The following lemma plays a key role in the proof of our main result (\Cref{main theorem12345}). In this lemma, we characterize all graphs $G$ on at most four vertices for which the equality $I_c(G)^{(k)} = I_c(G)^k$ holds. We note that there are exactly seven graphs on four vertices without any isolated vertex, which are depicted in \Cref{fig:graphs4vertices}.
	
	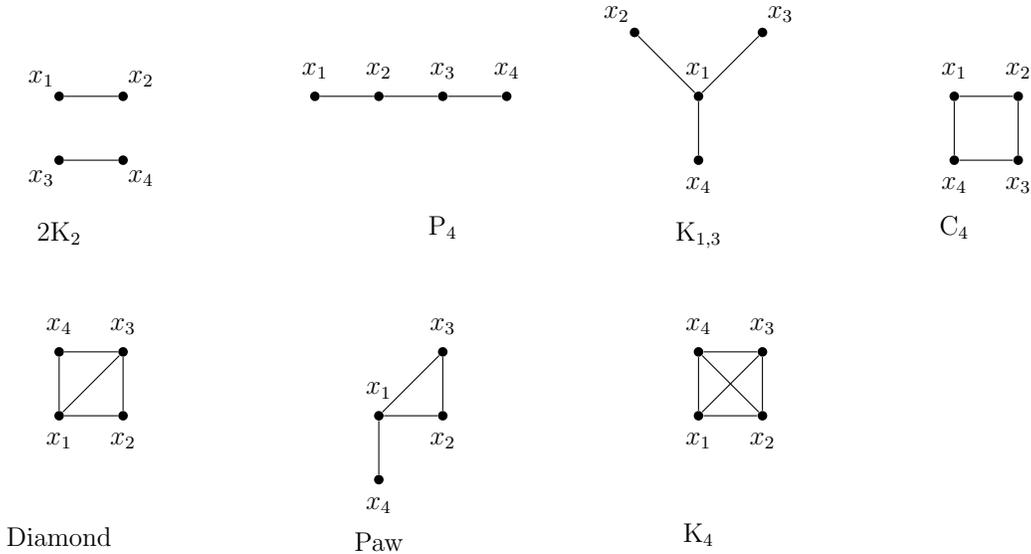
\begin{figure}[htbp]
		\centering
		\begin{tikzpicture}[scale=0.85, transform shape, every node/.style={circle,inner sep=1.5pt, fill=black}]
			
			% ----------------- Top Row -----------------
			% Graph 1: 2K2
			\node[label=above left:{$x_1$}] (G1v1) at (0,0) {};
			\node[label=above right:{$x_2$}] (G1v2) at (1,0) {};
			\node[label=below left:{$x_3$}] (G1v3) at (0,-1) {};
			\node[label=below right:{$x_4$}] (G1v4) at (1,-1) {};
			\draw (G1v1) -- (G1v2);
			\draw (G1v3) -- (G1v4);
			\node[below=.6cm of G1v3, draw=none, fill=none] {2K$_2$};
			
			% Graph 2: P4
			\node[label=above:{$x_1$}] (P4v1) at (4,0) {};
			\node[label=above:{$x_2$}] (P4v2) at (5,0) {};
			\node[label=above:{$x_3$}] (P4v3) at (6,0) {};
			\node[label=above:{$x_4$}] (P4v4) at (7,0) {};
			\draw (P4v1) -- (P4v2) -- (P4v3) -- (P4v4);
			\node[below=1.6cm of P4v3, draw=none, fill=none] {P$_4$};
			
			% Graph 3: K1,3
			\node[label=above:{$x_1$}] (C1) at (10,0) {};
			\node[label=above left:{$x_2$}] (C2) at (9,1) {};
			\node[label=above right:{$x_3$}] (C3) at (11,1) {};
			\node[label=below:{$x_4$}] (C4) at (10,-1) {};
			\draw (C1) -- (C2) (C1) -- (C3) (C1) -- (C4);
			\node[below=.6cm of C4, draw=none, fill=none] {K$_{1,3}$};
			
			% Graph 4: C4
			\node[label=above:{$x_1$}] (C4v1) at (14,0) {};
			\node[label=above:{$x_2$}] (C4v2) at (15,0) {};
			\node[label=below:{$x_4$}] (C4v4) at (14,-1) {};
			\node[label=below:{$x_3$}] (C4v3) at (15,-1) {};
			\draw (C4v1) -- (C4v2) -- (C4v3) -- (C4v4) -- (C4v1);
			\node[below=.6cm of C4v4, draw=none, fill=none] {C$_4$};
			
			% ----------------- Bottom Row -----------------
			% Graph 5: Diamond (K4-e)
			\node[label=below:{$x_1$}] (D1) at (0,-5) {};
			\node[label=below:{$x_2$}] (D2) at (1,-5) {};
			\node[label=above:{$x_3$}] (D3) at (1,-4) {};
			\node[label=above:{$x_4$}] (D4) at (0,-4) {};
			\draw (D1)--(D2)--(D3)--(D4)--(D1) (D1)--(D3);
			\node[below=.9cm of D1, draw=none, fill=none] {Diamond};
			
			% Graph 6: Paw
			\node[label=above:{$x_1$}] (Paw1) at (5,-5) {};
			\node[label=below:{$x_2$}] (Paw2) at (6,-5) {};
			\node[label=above:{$x_3$}] (Paw3) at (6,-4) {};
			\node[label=below:{$x_4$}] (Paw4) at (5,-6) {};
			\draw (Paw1)--(Paw2)--(Paw3)--(Paw1) (Paw1)--(Paw4);
			\node[below=.4cm of Paw4, draw=none, fill=none] {Paw};
			
			% Graph 7: K4
			\node[label=below:{$x_1$}] (K1) at (10,-5) {};
			\node[label=below:{$x_2$}] (K2) at (11,-5) {};
			\node[label=above:{$x_3$}] (K3) at (11,-4) {};
			\node[label=above:{$x_4$}] (K4) at (10,-4) {};
			\draw (K1)--(K2)--(K3)--(K4)--(K1) (K1)--(K3) (K2)--(K4);
			\node[below=1.4cm of K1, draw=none, fill=none] {K$_4$};
			
		\end{tikzpicture}
		\caption{All non-isomorphic graphs on four vertices without isolated vertices.}
		\label{fig:graphs4vertices}
	\end{figure}
	
	\begin{lemma}\label{vertex lemma}
		Let $G$ be a graph on at most four vertices with no isolated vertex and at least one edge, and let $k\ge 2$ be an integer. Then 
		\[
		I_c(G)^{(k)}=I_c(G)^{k}
		\]
		if and only if $G$ is either a $K_2$, $K_3$, $P_3$, $2K_2$, $P_4$ or $C_4$.
	\end{lemma}
	\begin{proof}
		If $|V(G)|\le 3$, then the only possibilities of $G$ are $K_2,K_3$ and $P_3$. In this case, it is easy to see that $I_c(G)$ is either the zero ideal or generated by variables. Thus, we always have $I_c(G)^{k}=I_c(G)^{(k)}$ when $|V(G)|\le 3$.
		
		Next, consider the case when $|V(G)|=4$. In this case, all possible graphs are depicted in \Cref{fig:graphs4vertices}, and by construction $I_c(G)$ is the same as the edge ideal of a graph $\til G$, where 
		\begin{align*}
			V(\til G)&=V(G),\\
			E(\til G)&=\{\{x,y\}\mid V(G)\setminus \{x,y\}\in E(G)\}.
		\end{align*} 
		It is easy to see that the graphs $\til G$ associated with $2K_2,P_4$ and $C_4$ are bipartite graph, and hence, by \cite[Theorem 5.9]{SimisVasconcelosVillarreal1994}, we have $I_c(G)^{k} = I_c(G)^{(k)}$. Now consider the graph $K_{1,3}$ in \Cref{fig:graphs4vertices}. Then, by \Cref{primary decomposition}, the primary decomposition of the complementary edge ideal is 
		\[
		I_c(K_{1,3})=\l x_2,x_3\r\cap \l x_2,x_4\r\cap \l x_3,x_4\r.
		\]
		Observe that $(x_2x_3)^{k-1}x_4\in I_c(K_{1,3})^{(k)}$, whereas $(x_2x_3)^{k-1}x_4\notin I_c(K_{1,3})^{k}$ since each monomial in $I_c(K_{1,3})^{k}$ has degree at least $2k$. Thus, $I_c(G)^{k} \neq I_c(G)^{(k)}$ if $G=K_{1,3}$.
		
		Next, consider the graph $G_1$, when $G_1$ is a {\it paw} in \Cref{fig:graphs4vertices}. In this case, by \Cref{primary decomposition}, the primary decomposition of the complementary edge ideal is 
		\[
		I_c(G_1)=\l x_2,x_4\r\cap \l x_3,x_4\r\cap \l x_1,x_2,x_3\r.
		\]
		By similar argument as above, we see that $(x_2x_4)^{k-1}x_3\in I_c(G_1)^{(k)}\setminus I_c(G_1)^k$. Thus, $I_c(G)^{k} \neq I_c(G)^{(k)}$ when $G$ is a paw.
		
		The remaining graphs are $K_4$ and the diamond graph in \Cref{fig:graphs4vertices}. Let $G_2$ denote the diamond graph. Then the primary decomposition for $K_4$ and $G_2$ are as follows:
		\begin{align*}
			I_c(K_4)&=\l x_1,x_2,x_4\r\cap \l x_1,x_3,x_4\r\cap \l x_1,x_2,x_3\r\cap \l x_2,x_3,x_4\r,\\
			I_c(G_2)&=\l x_2,x_4\r\cap \l x_1,x_3,x_4\r\cap \l x_1,x_2,x_3\r.
		\end{align*}
		Proceeding as before, we see that $(x_1x_2)^{k-1}x_3\in I_c(K_4)^{(k-1)}\setminus I_c(K_4)^{k-1}$ and $(x_1x_2)^{k-1}x_4\in I_c(G_1)^{(k-1)}\setminus I_c(G_1)^{k-1}$. This completes the proof of the lemma.
	\end{proof}
	
	We are now ready to prove the main result of this article. Throughout, we adopt the notation $\deg(m)$ to denote the $\mathbb N$-graded degree of a monomial $m$ in the polynomial ring $R$.
	
	\begin{theorem}\label{main theorem12345}
		Let $\H$ be an $(n-2)$-uniform clutter with $|V(\H)|=n$. Let $G$ denote the associated graph such that $I(\H)=I_c(G)$. Then the following are equivalent.
		\begin{enumerate}
			\item $I(\H)^{(k)}=I(\H)^k$ for all $k\geq 1$;
			\item $I(\H)^{(k)}=I(\H)^k$ for some $k\geq 2$;
			\item $\H$ satisfies the MFMC property;
			
			\item $\H$ satisfies the packing property; 
			
			\item $G$ is either a $K_2$, $K_3$, $P_3$, $2K_2$, $P_4$ or $C_4$ with (possibly) some isolated vertices.
		\end{enumerate}
	\end{theorem}
	
	\begin{proof}
		The equivalence $(1)\iff (3)$ follows from \cite[Theorem 14.3.6]{RHV}. Moreover, $(1)\Rightarrow (2)$ is easy to see. Furthermore, $(5)\Rightarrow (1)$ follows from \Cref{isolated lemma} and \Cref{vertex lemma}. 
		
		To prove $(2)\Rightarrow (5)$, we take $V(\H)=\{x_1,\ldots,x_n\}$ and $I(\H)^{(k)}=I(\H)^k$. Let $G$ be the associated graph such that $I(\H)=I_c(G)$. Then we have $I_c(G)^k=I_c(G)^{(k)}$. If $|V(G)|\le 4$, then we obtain $(5)$ by using \Cref{isolated lemma} and \Cref{vertex lemma}. Therefore, we may assume that $|V(G)|\ge 5$. First consider the case when $G$ is a star graph, i.e., suppose $E(G)=\{\{x_1,x_i\}\mid i\in[n]\setminus\{1\}\}$. In this case, by \Cref{primary decomposition}, we have 
		\[
		I_c(G)=\bigcap_{\underset{r,s>1}{r\neq s}} \l x_r,x_s\r.
		\]
		Observe that $\prod_{i=2}^nx_i\in I_c(G)^{(2)}$, and if $e\in E(G)$, then $m=\x_{\bar e}^{k-2}\prod_{i=2}^nx_i\in I_c(G)^{k-2}I_c(G)^{(2)}\subseteq I_c(G)^{(k)} $, where $\deg(m)=(n-2)(k-2)+n-1$. By assumption, $m\in I_c(G)^{k}$, and hence, $k(n-2)\le \deg(m)$. Consequently, $n\le 3$, a contradiction. Thus, from now on, we may assume $G$ is not a star graph and $|V(G)|\ge 5$. Let us choose an edge $e\in E(G)$.
		\medskip 
		
		\noindent
		{\bf Claim:} $\x_{\bar e}^{k-2}\left(\prod_{i=1}^nx_i\right)\in I_c(G)^{(k)}$.
		
		\medskip 
		
		\noindent
		{\bf Proof of Claim:} Since $I_c(G)^t \subseteq I_c(G)^{(t)}$ for every $t \ge 1$, it suffices to show that 
		\[
		\prod_{i=1}^n x_i \in I_c(G)^{(2)}.
		\]
		By \Cref{primary decomposition}, we have 
		\[
		I_c(G)
		=
		\Bigg(\bigcap_{e \in E(G^c)} \p_e\Bigg)
		\;\cap\;
		\Bigg(\bigcap_{G[A] \cong K_3} \p_A\Bigg).
		\]
		Note that for each $e \in E(G^c)$ and each subset $A \subseteq V(G)$ with $G[A] \cong K_3$, there exists at least one squarefree monomial of degree $2$ in both $\p_e^2$ and $\p_A^2$. Therefore, to show that $\prod_{i=1}^n x_i \in I_c(G)^{(2)}$, it suffices to verify that for every vertex $x$ of $G$, either $x \in e$ for some $e \in E(G^c)$ or $x \in A$, where $A \subseteq V(G)$ with $G[A] \cong K_3$. This indeed holds, since $G$ is not a star graph. Hence, the proof of the claim is complete.
		
		Now, since $I_c(G)^{k} = I_c(G)^{(k)}$, we have $\x_{\bar e}^{k-2}\left(\prod_{i=1}^nx_i\right)\in I_c(G)^{k}$. However, the minimal generators of $I_c(G)^{k}$ have degree $(n-2)k$. Therefore, we must have $k(n-2)\le n+(k-2)(n-2)$. But this implies $n\le 4$, a contradiction to our hypothesis. This completes the proof of $(2)\Rightarrow (5)$.
		
		Observe that to complete the proof of the theorem, it is enough to show $(4)\iff (5)$. Now, suppose $(5)$ holds. Then $I_c(G)=mJ$ for some monomial $m$ and a monomial ideal $J$ with support disjoint from $m$, where either $J=S$ or $J$ is generated by variables, or $J$ is an edge ideal of a bipartite graph. Thus $(4)$ follows from \Cref{packing lemma} and \cite[Proposition 4.27]{grv2009}. Next, we assume $(4)$, i.e., the clutter $\H$ satisfies the packing property. Let $G$ be the associated graph of $\H$ such that $I_c(G)=I(\H)$. If $G$ has an isolated vertex $x$, consider the graph $\ha{G}=G\setminus \{x\}$. Then, $\widehat{\H}_x=\H$, and hence, by \Cref{packing lemma}, $\widehat{\H}$ also has the packing property, where $I(\widehat{\H})=I_{c}(\ha{G})$. Thus, a priori, we may assume that $G$ has no isolated vertex. Therefore, by \Cref{primary decomposition}, we have
		\[
		\height(I_{c}(G))\ge 2.
		\]
		Since $\H$ satisfies the packing property, there should exist at least two monomials $g_1,g_2\in \G(I_{c}(G))$ with disjoint support of variables. Therefore, we have $\deg(g_1)+\deg(g_2)=2n-4\leq n$, which gives $n\leq 4$.
		 For $n=4$, $I_{c}(G)$ is an edge ideal of a graph with packing property, and thus, that graph must be bipartite by \cite[Proposition 4.27]{grv2009}. Then from \Cref{fig:graphs4vertices}, one can easily verify that $G$ must be $2K_2$, $P_4$ or $C_4$. For $n<4$, it is straightforward that $I_{c}(G)$ has the packing property for any graph $G$. Hence, $G$ can be $K_2,K_3$ or $P_3$. This completes the proof of the theorem.
	\end{proof}
	
	An immediate consequence of \Cref{main theorem12345} is the following.
	
	\begin{corollary}
		Let $\H$ be an $(n-2)$-uniform clutter with $|V(\H)|=n$. Then $\H$ satisfies the Conforti–Cornu\'ejols \Cref{packing conjecture}.
	\end{corollary}
	
	\begin{remark}\label{last remark}
		Let $G$ be a $2$-uniform clutter, i.e., a graph. Then unlike $(n-2)$-uniform clutter there may exist some integer $k\ge 2$ such that $I(G)^{(k)}=I(G)^{k}$ but $G$ fails to satisfy the MFMC property. Indeed, as an example, one can take $G=C_5$, the five-cycle, and observe that $I(G)^{(2)}=I(G)^{2}$ but $G$ does not satisfy the MFMC property, since $G$ is not bipartite \cite[cf. Proposition 4.27]{grv2009}.
	\end{remark}
	
	\subsection*{Comparision with other clutters:}
	
	In the this part of the section, we construct, for each of the following classes, $(n-2)$-uniform clutters that do not belong to that class: balanced clutters, binary clutters, dyadic clutters, and the $2$-partitionable clutters considered in \cite{FloresMendezGitlerReyes2008}. 
	Note that for each of these classes, the Conforti-Cornu\'ejols conjecture is verified, as mentioned in \Cref{intro}.

	\begin{enumerate}
		\item[$\bullet$] A clutter $C$ is \emph{balanced} if its incidence matrix contains no square submatrix of odd order with exactly two $1$'s in each row and column, that is, if $C$ has no special odd cycles \cite[Definition 6.5.15]{RHV}. 
		For $n=4$, the $2$-uniform clutter $\H$ with $V(\H)=\{x_1,x_2,x_3,x_4\}$ and edge set $E(\H)=\{x_1x_2, x_2x_3, x_1x_3, x_1x_4\}$ is not balanced, since its incidence matrix contains the submatrix corresponding to the vertex set $\{x_1,x_2,x_3\}$ with exactly two $1$'s in each row and column. 
		Similarly, for $n\ge 5$, the $(n-2)$-uniform clutter $\H$ with $V(\H)=\{x_1,\dots,x_n\}$ and edge set 
		\[
		E(\H)=\Big\{\prod_{\underset{i\neq 3}{i=1}}^{n-1}x_i, \prod_{i=2}^{n-1}x_i, \prod_{\underset{i\neq 2}{i=1}}^{n-1}x_i, \prod_{i=3}^{n}x_i\Big\}
		\]
		is not balanced, as its incidence matrix contains the submatrix on the vertex set $\{x_1,x_2,x_3\}$ with exactly two $1$'s in each row and column.
		
		\item[$\bullet$] A clutter $\H$ is said to be \emph{binary} if each edge and each minimal vertex cover of $\H$ intersect in an odd number of vertices. 
		Let $\H$ be the $(n-2)$-uniform clutter on the vertex set $V(\H)=\{x_1,\dots,x_n\}$ with edge set $\{e_i \mid i\in [n-1]\}$, where 
		\[
		e_i = \{x_1, \dots, x_{i-1}, \widehat{x_i}, \widehat{x_{i+1}}, x_{i+2}, \dots, x_n\}.
		\]
		Observe that $I(\H)$ is the complementary edge ideal of $P_n$, and by \Cref{primary decomposition}, $\{x_1,x_n\}$ is a minimal vertex cover of $\H$. 
		Since $|e_2 \cap \{x_1,x_n\}| = 2$, it follows that $\H$ is not a binary clutter.
		
		\item[$\bullet$] A clutter $\H$ is said to be \emph{dyadic} if each edge and each minimal vertex cover of $\H$ intersect in at most two vertices. 
		Let $\H$ be the $(n-2)$-uniform clutter on the vertex set $V(\H)=\{x_1, \dots, x_n\}$ with edge set 
		\[
		\{e_i,e_n \mid i \in [n-1]\},\] where  $e_i = \{x_1, \dots, x_{i-1}, \widehat{x_i}, \widehat{x_{i+1}}, x_{i+2}, \dots, x_n\}$, $e_n = \{x_1, \dots, x_{n-1}\}$. Observe that $I(\H)$ is the complementary edge ideal of the graph with edge set 
		\(\{\{x_i, x_{i+1}\}, \{x_{n-2}, x_n\} \mid i \in [n-1]\}\), and by \Cref{primary decomposition}, $\{x_{n-2}, x_{n-1}, x_n\}$ is a minimal vertex cover of $\H$. 
		Since $|e_1 \cap \{x_{n-2}, x_{n-1}, x_n\}| = 3$, it follows that $\H$ is not a dyadic clutter.
		
		\item[$\bullet$] It is clear that not every $(n-2)$-uniform clutter is a $2$-partitionable clutter of the form $Q_{pq}^F$ defined in \cite{FloresMendezGitlerReyes2008}, for which the Conforti-Cornu\'ejols conjecture is known to hold, since $|V(Q_{pq}^F)|$ is always even.
	\end{enumerate}
	
    \subsection*{Application to Linear Programming (LP) duality problems:} In the final part of this section, we present an application of our results to the theory of LP-duality problems. We begin with some observations on the incidence matrices of clutters.
	
	Let $\H$ be a clutter with $|V(\H)|=n$ and $|E(\H)|=m$. Denote by $A$ the $m\times n$ (edge–vertex) incidence matrix of $\H$. For $r\ge 1$, define a new clutter $\H_{y_1\cdots y_r}$ with vertex set $V(\H)\cup\{y_1,\ldots,y_r\}$ and edge set 
	\[
	E(\H_{y_1\cdots y_r})=\{e\cup\{y_1,\ldots,y_r\}\mid e\in E(\H)\}.
	\]
	Then the incidence matrix $\widetilde{A}_r$ of $\H_{y_1\cdots y_r}$ is the block matrix
	\begin{align}\label{incidence1}
		\widetilde{A}_r=\left(
		\begin{array}{c|c}
			A & B
		\end{array}
		\right),
	\end{align}
	of order $m\times (n+r)$, where $B$ is an $m\times r$ matrix with all entries equal to $1$. In the sequel, we shall use this description of the incidence matrix of $\H_{y_1\cdots y_r}$ to illustrate the structure of incidence matrices of all $(n-2)$-uniform clutters satisfying the MFMC property. From now on, for any $m\times n$ matrix $A$ and any non-negative integer $r$, we denote by $\widetilde{A}_r$ the matrix specified in \Cref{incidence1}, with the convention that $\widetilde{A}_0=A$.
	
	Let $G$ be one of the graphs $K_2$, $K_3$, $P_3$, $2K_2$, $P_4$, or $C_4$. Denote by $\H_G$ the clutter associated to $G$ such that $I(\H_G)=I_c(G)$. Then the incidence matrices of the clutters $\H_G$ corresponding to these graphs are as follows.

	{\footnotesize
	\begin{align}\label{matrices}
		\begin{pmatrix}
		\textcolor{white}{0}	
		\end{pmatrix}_{0\times 2},\begin{pmatrix}
		1 &0 &0\\
		0&1&0\\
		0&0&1
		\end{pmatrix}_{3\times 3},\begin{pmatrix}
		0&0&1\\
		1&0&0
		\end{pmatrix}_{2\times 3}, \begin{pmatrix}
		1&1&0&0\\
		0&0&1&1
		\end{pmatrix}_{2\times 4}, \begin{pmatrix}
		0&0&1&1\\
		1&0&0&1\\
		1&1&0&0
		\end{pmatrix}_{3\times 4}, \begin{pmatrix}
		0&0&1&1\\
		1&0&0&1\\
		1&1&0&0\\
		1&1&1&0
		\end{pmatrix}_{4\times 4},
	\end{align} 
}
where $	\begin{pmatrix}
	\textcolor{white}{0}	
\end{pmatrix}_{0\times 2}$ denotes the incidence matrix of the clutter on two vertices with no edges.

Let $M$ be any matrix of order $m\times n$. Consider the following linear programming problems:
\begin{align}\label{lpp1}
	\begin{split}
		\text{minimize } & \alpha\cdot x,\\
		\text{subject to } & Mx\ge \mathbf{1},\; x\in\mathbb{N}^n,
	\end{split}
\end{align}
and
\begin{align}\label{lpp2}
	\begin{split}
		\text{maximize } & y\cdot \mathbf{1},\\
		\text{subject to } & yM\le \alpha,\; y\in\mathbb{N}^n,
	\end{split}
\end{align}
where $\mathbf{1}$ is the column vector whose entries are all equal to $1$, and $\alpha$ is a nonnegative integral vector.  
Let $\varphi_{\alpha}(M)$ and $\psi_{\alpha}(M)$ denote the optimal values of the linear programming problems in \Cref{lpp1} and \Cref{lpp2}, respectively.  
As an application of \Cref{main theorem12345}, we now provide an explicit solution to the following linear programming problem.
	
	\begin{theorem}\label{lpp theorem}
		Let $M$ be any $m\times n$ matrix with entries $0$ and $1$. If each row sum of $M$ is $n-2$, then for each $\alpha\in\mathbb N^n$, $\varphi_{\alpha}(M)=\psi_{\alpha}(M)$ if and only if $n\le 4$, and $M=\widehat{A}_r$, where $A$ is any of the matrices in \Cref{matrices} and $r$ is a non-negative integer.
	\end{theorem}
	\begin{proof}
		As each row sum of $M$ is $n-2$, we see that $M$ is the incidence matrix of an $(n-2)$-uniform clutter. Thus, the result follows from \Cref{main theorem12345} and the discussion above.
	\end{proof}
	
	\section*{Acknowledgements}
	The authors are deeply grateful to Rafael Villarreal for his insightful comments and suggestions, which greatly improved the quality of this work. They also thank Antonino Ficarra and Mehrdad Nasernejad for their valuable feedback. Roy acknowledges support from a Postdoctoral Fellowship at the Chennai Mathematical Institute and a grant from the Infosys Foundation.

	\subsection*{Data availability statement} Data sharing does not apply to this article as no new data were created or analyzed in this study.
	
	\subsection*{Conflict of interest} The authors declare that they have no known competing financial interests or personal relationships that could have appeared to influence the work reported in this paper.

	\bibliographystyle{abbrv}
	\bibliography{ref}

@article{FulkersonHoffmanOppenheim1974,
  author    = {D. R. Fulkerson and A. J. Hoffman and R. Oppenheim},
  title     = {On balanced matrices},
  journal   = {Mathematical Programming Study},
  volume    = {1},
  year      = {1974},
  pages     = {120--132}
}

@article{FicarraMoradi2025SymbolicPowersPolymatroidal,
  author    = {Antonino Ficarra and Somayeh Moradi},
  title     = {Symbolic Powers of Polymatroidal Ideals},
  journal   = {Journal of Pure and Applied Algebra},
  volume    = {229},
  number    = {10},
  pages     = {108082},
  year      = {2025},
  doi       = {10.1016/j.jpaa.2025.108082},
  url       = {https://doi.org/10.1016/j.jpaa.2025.108082},
}

@article{CornuejolsMargotGuenin2000,
  author    = {G. Cornu\'ejols and F. Margot and B. Guenin},
  title     = {The packing property},
  journal   = {Mathematical Programming, Series A},
  volume    = {89},
  year      = {2000},
  pages     = {113--126}
}

@article{FloresMendezGitlerReyes2008,
  author    = {A. Flores-M{\'e}ndez and I. Gitler and E. Reyes},
  title     = {On 2-partitionable clutters and the {MFMC} property},
  journal   = {Advances and Applications in Discrete Mathematics},
  volume    = {2},
  year      = {2008},
  pages     = {59--84}
}

@article{Seymour1977,
  author    = {P. D. Seymour},
  title     = {The matroids with the max-flow min-cut property},
  journal   = {Journal of Combinatorial Theory, Series B},
  volume    = {23},
  year      = {1977},
  pages     = {189--222}
}

@article {grv2009,
    AUTHOR = {Gitler, Isidoro and Reyes, Enrique and Villarreal, Rafael H.},
     TITLE = {Blowup algebras of square-free monomial ideals and some links
              to combinatorial optimization problems},
   JOURNAL = {Rocky Mountain J. Math.},
  FJOURNAL = {The Rocky Mountain Journal of Mathematics},
    VOLUME = {39},
      YEAR = {2009},
    NUMBER = {1},
     PAGES = {71--102},
      ISSN = {0035-7596,1945-3795},
   MRCLASS = {13A30 (13B22 13F20 13H10 52B20)},
  MRNUMBER = {2476802},
MRREVIEWER = {John\ Clark},
       DOI = {10.1216/RMJ-2009-39-1-71},
       URL = {https://doi.org/10.1216/RMJ-2009-39-1-71},
}

@article {RTY11,
    AUTHOR = {Rinaldo, Giancarlo and Terai, Naoki and Yoshida, Ken-ichi},
     TITLE = {Cohen-{M}acaulayness for symbolic power ideals of edge ideals},
   JOURNAL = {J. Algebra},
  FJOURNAL = {Journal of Algebra},
    VOLUME = {347},
      YEAR = {2011},
     PAGES = {1--22},
      ISSN = {0021-8693,1090-266X},
   MRCLASS = {13F20 (13F55 13H10)},
  MRNUMBER = {2846393},
MRREVIEWER = {Louiza\ Fouli},
       DOI = {10.1016/j.jalgebra.2011.09.007},
       URL = {https://doi.org/10.1016/j.jalgebra.2011.09.007},
}

@article{Mendez2025Symbolic,
  author    = {Fernando O. Méndez and Maria Vaz Pinto and Rafael H. Villarreal},
  title     = {Symbolic powers: Simis and weighted monomial ideals},
  journal   = {Journal of Algebra and Its Applications},
  year      = {2025},
  volume    = {25},
  number    = {4},
  pages     = {2541001},
  doi       = {10.1142/S0219498825410014},
  publisher = {World Scientific Publishing},
  url       = {https://www.worldscientific.com/doi/10.1142/S0219498825410014}
}

@article{Gitler2007Rees,
  author = {Isidoro Gitler and Carlos E. Valencia and Rafael H. Villarreal},
  title = {A note on Rees algebras and the MFMC property},
  journal = {Beiträge zur Algebra und Geometrie},
  volume = {48},
  number = {1},
  pages = {141--150},
  year = {2007}
}

@book {RHV,
    AUTHOR = {Villarreal, Rafael H.},
     TITLE = {Monomial algebras},
    SERIES = {Monographs and Research Notes in Mathematics},
   EDITION = {Second},
 PUBLISHER = {CRC Press, Boca Raton, FL},
      YEAR = {2015},
     PAGES = {xviii+686},
      ISBN = {978-1-4822-3469-5},
   MRCLASS = {13-02 (05C65 05E15 90C27)},
  MRNUMBER = {3362802},
MRREVIEWER = {Siamak\ Yassemi},
}

@article{HochstattlerNasernejad2024,
  author    = {Winfried Hochstättler and Mehrdad Nasernejad},
  title     = {A Classification of Mengerian 4-uniform Hypergraphs Derived from Graphs},
  journal   = {Ars Combinatoria},
  volume    = {161},
  pages     = {49--59},
  year      = {2024},
  doi       = {10.61091/ars161-04},
  url       = {https://www.researchgate.net/publication/387849932_A_Classification_of_Mengerian_4-uniform_Hypergraphs_Derived_from_Graphs}
}

@article{SimisVasconcelosVillarreal1994,
  author    = {Antonio Simis and Wolmer Vasconcelos and Rafael Villarreal},
  title     = {On the ideal theory of graphs},
  journal   = {Journal of Algebra},
  volume    = {167},
  pages     = {389--416},
  year      = {1994},
  doi       = {10.1006/jabr.1994.1140},
  url       = {https://doi.org/10.1006/jabr.1994.1140}
}

@article{AlilooeeBanerjee2021,
  author    = {Ali Alilooee and Arindam Banerjee},
  title     = {Packing properties of cubic square-free monomial ideals},
  journal   = {Journal of Algebraic Combinatorics},
  volume    = {54},
  number    = {5},
  pages     = {803--813},
  year      = {2021},
  doi       = {10.1007/s10801-021-01020-2},
  url       = {https://doi.org/10.1007/s10801-021-01020-2}
}

@incollection{DaoDeGriHuBet2018,
  author    = {Dao, Hailong and De Stefani, Alessandro and Grifo, Eloísa and Huneke, Craig and Núñez-Betancourt, Luis},
  title     = {Symbolic powers of ideals},
  booktitle = {Singularities and foliations. geometry, topology and applications},
  series    = {Springer Proc. Math. Stat.},
  volume    = {222},
  pages     = {387--432},
  publisher = {Springer},
  address   = {Cham},
  year      = {2018},
  }

@incollection{ConfortiCornuejols1993,
  author    = {Gérard Conforti and Michele Cornuéjols},
  title     = {Clutters that Pack and the Max Flow Min Cut Property: A Conjecture},
  booktitle = {The Fourth Bellairs Workshop on Combinatorial Optimization},
  editor    = {W.R. Pulleyblank and F.B. Shepherd},
  year      = {1993},
  publisher = {Bellairs Institute},
  doi       = {10.21236/ADA277340}
}

@article{SzembergSzpond2017,
  author    = {Tomasz Szemberg and Justyna Szpond},
  title     = {On the containment problem},
  journal   = {Rendiconti del Circolo Matematico di Palermo Series 2},
  volume    = {66},
  number    = {2},
  pages     = {233--245},
  year      = {2017},
  doi       = {10.1007/s12215-017-0299-3},
  url       = {https://doi.org/10.1007/s12215-017-0299-3}
}

@article{FicarraMoradiSR2025,
  title={Stanley-{R}eisner ideals with linear powers},
  author={Antonino Ficarra and Somayeh Moradi},
  journal={arXiv:2508.10354},
  year={2025}
}

@article{FicarraMoradiRees2025,
  title={Rees algebras of complementary edge ideals},
  author={Antonino Ficarra and Somayeh Moradi},
  journal={arXiv:2509.18048},
  year={2025}
}

@article{FicarraMoradiComp2025,
  title={Complementary edge ideals},
  author={Antonino Ficarra and Somayeh Moradi},
  journal={arXiv:2508.10870},
  year={2025}
}

@article{HibiQureshiMadani2025,
  title={Complementary edge ideals},
  author={Takayuki Hibi and Ayesha Asloob Qureshi and Sara Saeedi Madani},
  journal={arXiv:2508.09837},
  year={2025}
}
\end{document}